\documentclass[12pt,english]{article}
\usepackage[T1]{fontenc}
\usepackage[latin9]{inputenc}
\usepackage{geometry}
\usepackage{mathtools}
\usepackage{amsmath}
\usepackage{amsthm}
\usepackage{amssymb}
\usepackage{microtype}

\makeatletter
\theoremstyle{plain}
\newtheorem{thm}{\protect\theoremname}
\theoremstyle{plain}
\newtheorem{lem}[thm]{\protect\lemmaname}
\theoremstyle{plain}
\newtheorem{cor}[thm]{\protect\corollaryname}
\theoremstyle{plain}
\newtheorem{prop}[thm]{\protect\propositionname}

\usepackage{bm}

\usepackage{tikz}
\usetikzlibrary{arrows}
\usepackage{hyperref}

\usepackage{colortbl}
\usepackage{xparse}
\usepackage{caption}

\usepackage{ytableau}

\usepackage{titlesec}
\titleformat{\section}
  {\normalfont\large\bfseries}{\thesection.}{.8ex}{} 
\titleformat{\subsection}
  {\normalfont\bfseries}{\thesubsection.}{.8ex}{} 

\usetikzlibrary{shapes}

\tikzstyle{pathdefault}=[draw, line width=1, solid, color=black]
\tikzstyle{nodedefault}=[circle, inner sep=1.5, fill=black]
\tikzstyle{empty}=[]
\tikzstyle{nodeellipsis}=[circle, inner sep=0.5, fill=black]
\tikzstyle{pathcolor1}=[draw, line width=1.3, densely dashed, color=red]
\tikzstyle{pathcolor2}=[draw, line width=1.6, densely dotted, color=blue]
\tikzstyle{pathcolorlight}=[draw, line width=1, dotted, color=lightgray]

\tikzstyle{arbpathcolor0}=[line width=1, dashdotted, color=black]
\tikzstyle{arbpathcolor1}=[line width=1, densely dashed, color=red]
\tikzstyle{arbpathdefault}=[line width=1, densely dotted, color=blue]

\newcounter{id}
\newcommand{\drawlinedotswithstyle}[4]{
 \def\x{{#3}}
 \def\y{{#4}}
 \tikzstyle{thispathstyle}=[#1]
 \tikzstyle{thisnodestyle}=[#2]
 \setcounter{id}{-1} 
 \foreach \j in {#3}{\stepcounter{id}} 
 \foreach \i in {1,...,\the\value{id}}{  
  \path[thispathstyle] (\x[\i],\y[\i]) --(\x[\i-1],\y[\i-1]); 
 }
 \foreach \i in {1,...,\the\value{id}}{  
  \node[thisnodestyle] at (\x[\i],\y[\i]) {}; 
 }
 \node[thisnodestyle] at (\x[0],\y[0]) {}; 
}

\DeclareDocumentCommand{\drawlinedots}{ O{pathdefault} O{nodedefault} m m}{\drawlinedotswithstyle{#1}{#2}{#3}{#4}}

\let\originalleft\left
\let\originalright\right
\renewcommand{\left}{\mathopen{}\mathclose\bgroup\originalleft}
\renewcommand{\right}{\aftergroup\egroup\originalright}

\definecolor{mhcblue}{HTML}{0077CC} 
\definecolor{davidsonred}{HTML}{AC1A2F} 

\definecolor{green}{RGB}{0, 180, 0}
\definecolor{yellow}{RGB}{180, 180, 0}

\makeatother

\usepackage{babel}
\providecommand{\corollaryname}{Corollary}
\providecommand{\lemmaname}{Lemma}
\providecommand{\theoremname}{Theorem}
\providecommand{\propositionname}{Proposition}

\newcounter{i}

\newcommand{\drawpermutation}[3][1]{\begin{tikzpicture}[scale=0.5,baseline=(O.base)]
\setcounter{i}{0}
\foreach \j in {#2} {
\stepcounter{i}
\draw (0.5*#1,\value{i}*#1) -- (#3*#1+0.5*#1,\value{i}*#1);
\draw (\value{i}*#1,0.5*#1) -- (\value{i}*#1,#3*#1+0.5*#1);
\node at (\value{i}*#1, 0) {\footnotesize$\j$};
\draw[fill] (\value{i}*#1, \j*#1) circle (0.2);
}
\node (O) at (#3*0.5*#1,#3*0.5*#1) {};
\end{tikzpicture}}

\renewcommand{\S}{\mathfrak{S}}

\begin{document}
\global\long\def\des{\operatorname{des}}%
\global\long\def\pk{\operatorname{pk}}%
\global\long\def\lpk{\operatorname{lpk}}%
\global\long\def\std{\operatorname{std}}%
\global\long\def\Des{\operatorname{Des}}%
\global\long\def\ipk{\operatorname{ipk}}%
\global\long\def\ilpk{\operatorname{ilpk}}%
\global\long\def\rpk{\operatorname{rpk}}%

\title{Fibonacci numbers, consecutive patterns, and inverse peaks}
\author{Justin M.\ Troyka and Yan Zhuang\\
Department of Mathematics and Computer Science\\
Davidson College\texttt{}~\\
\texttt{\{jutroyka, yazhuang\}@davidson.edu}}
\maketitle
\begin{abstract}
We give multiple proofs of two formulas concerning the enumeration
of permutations avoiding a monotone consecutive pattern with a certain
value for the inverse peak number or inverse left peak number statistic.
The enumeration in both cases is given by a sequence related to Fibonacci
numbers. We also show that there is exactly one permutation whose
inverse peak number is zero among all permutations with any fixed
descent composition, and we give a few elementary consequences of
this fact. Our proofs involve generating functions, symmetric functions, regular expressions, and monomino-domino tilings.
\end{abstract}
\textbf{\small{}Keywords: }{\small{}permutations, peaks, consecutive
patterns, Fibonacci numbers}{\let\thefootnote\relax}\bigskip{}

\section{Introduction}

Let $\mathfrak{S}_{n}$ denote the symmetric group of permutations
on the set $[n]\coloneqq\{1,2,\dots,n\}$. 
We write permutations in one-line notation\textemdash that is, $\pi=\pi_{1}\pi_{2}\cdots\pi_{n}$\textemdash and
the $\pi_{i}$ are called \textit{letters} of $\pi$. The \textit{length}
of $\pi$ is the number of letters in $\pi$, so that $\pi$ has length
$n$ whenever $\pi\in\mathfrak{S}_{n}$. 

For a sequence of $n$
distinct integers $w$, the \textit{standardization} of $w$\textemdash denoted
$\std(w)$\textemdash is defined to be the permutation in $\mathfrak{S}_{n}$
obtained by replacing the smallest entry of $w$ with 1, the second
smallest with 2, and so on. As an example, we have $\std(83614)=52413$.
Given permutations $\pi\in\mathfrak{S}_{n}$ and $\sigma\in\mathfrak{S}_{m}$,
we say that $\pi$ \textit{contains} $\sigma$ (as a \textit{consecutive
pattern}) if $\std(\pi_{i}\pi_{i+1}\cdots\pi_{i+m-1})=\sigma$ for
some $i\in[n-m+1]$. If $\pi$ does not contain $\sigma$, then we
say that $\pi$ \textit{avoids} $\sigma$ (as a consecutive pattern).
Let $\mathfrak{S}_{n}(\sigma)$ denote the subset of permutations
in $\mathfrak{S}_{n}$ which avoid $\sigma$.

This paper supplements the recent paper \cite{Zhuang2021}, in which
the second author proves a lifting of the Goulden\textendash Jackson
cluster method for permutations\textemdash a standard tool in the
study of consecutive patterns\textemdash to the Malvenuto\textendash Reutenauer
algebra. By applying various homomorphisms to this generalized cluster
method, the second author obtains in \cite{Zhuang2021} various specializations which
allows one to count permutations avoiding prescribed consecutive patterns
while keeping track of certain ``inverse'' permutation statistics.
These statistics include the inverse descent number, the inverse peak
number, and the inverse left peak number.

In this paper, we prove two observations which were left unproven
in \cite{Zhuang2021}; these are stated in Theorems \ref{t-ipk} and
\ref{t-ilpk} below. Both theorems concern the enumeration of permutations
avoiding a monotone pattern with a certain value for the inverse peak
number or inverse left peak number statistic, and the enumeration
in both cases is given by a sequence related to Fibonacci numbers.

Let us establish a few more definitions. For a permutation $\pi$
in $\mathfrak{S}_{n}$, we say that $i\in\{2,3,\dots,n-1\}$ is a
\textit{peak} of $\pi$ if $\pi_{i-1}<\pi_{i}>\pi_{i+1}$, and we
say that $i\in[n-1]$ is a \textit{left peak} of $\pi$ if either
$i$ is a peak of $\pi$ or if $i=1$ and $\pi_{1}>\pi_{2}$. Let
$\ipk(\pi)$ be the number of peaks of $\pi^{-1}$, and let $\ilpk(\pi)$
be the number of left peaks of $\pi^{-1}$. For example, if $\pi=23568714$,
then $\pi^{-1}=71283465$ and we have $\ipk(\pi)=2$ and $\ilpk(\pi)=3$.

The \textit{Fibonacci sequence of order $k$} (also called the \textit{$k$-generalized
Fibonacci sequence}), denoted $\{f_{n}^{(k)}\}_{n\geq0}$, is defined
by the recursion 
\[
f_{n}^{(k)}=f_{n-1}^{(k)}+f_{n-2}^{(k)}+\cdots+f_{n-k}^{(k)}
\]
with $f_{0}^{(k)}\coloneqq1$ (and where we treat $f_{n}^{(k)}$ as
0 for $n<0$). Hence, the Fibonacci sequence of order two is the usual
Fibonacci sequence \cite[A000045]{oeis} and the Fibonacci sequence
of order three is the \textit{tribonacci sequence} \cite[A000073]{oeis}.\footnote{Note that the OEIS uses a different indexing for these sequences.}
The number $f_{n}^{(k)}$ counts tilings of a $1\times n$ rectangle
with tiles of size $1\times j$ where $j$ ranges from 1 to $k$ \cite[Section 3.4]{Benjamin2003};
equivalently, $f_{n}^{(k)}$ is the number of compositions of $n$
with no parts greater than $k$. Several other combinatorial interpretations
for $f_{n}^{(k)}$ are known; see the references in \cite[A092921]{oeis}.
Theorem \ref{t-ipk} gives another combinatorial interpretation for
$f_{n}^{(k)}$.
\begin{thm}[{\cite[Claim 4.6]{Zhuang2021}}]
\label{t-ipk}Let $n\geq1$ and $m\geq3$. The number of permutations
$\pi$ in $\mathfrak{S}_{n}(12\cdots m)$ with $\ipk(\pi)=0$\textemdash that
is, such that $\pi^{-1}$ has no peaks\textemdash is equal to $f_{n}^{(m-1)}$.
\end{thm}

The next theorem concerns OEIS sequence A080145 \cite[A080145]{oeis}.
As far as we know, this is the first combinatorial interpretation
related to permutation enumeration for the numbers in \cite[A080145]{oeis}.
Here, $f_{n}\coloneqq f_{n}^{(2)}$ denotes the $n$th Fibonacci number.
\begin{thm}[{\cite[Claim 4.9]{Zhuang2021}}]
\label{t-ilpk} Let $n\geq1$. The number of permutations $\pi$
in $\mathfrak{S}_{n}(321)$ with $\ilpk(\pi)=1$\textemdash that is,
such that $\pi^{-1}$ has exactly one left peak\textemdash is equal
to 
\[
\sum_{i=1}^{n-1}\sum_{k=1}^{i}f_{k-1}f_{k}=f_{n-1}f_{n}-\left\lfloor \frac{n+1}{2}\right\rfloor .
\]
\end{thm}

This paper is organized as follows. In Section 2, we prove both theorems
using generating functions derived in \cite{Zhuang2021}. 
Sections 3 and 4 provide alternative proofs for these theorems which we believe are more enlightening.
In Section 3, we show that there is exactly one permutation $\pi$ with $\ipk(\pi)=0$
among all permutations with any fixed ``descent composition'' (Theorem
\ref{t-ipkdescomp}), and Theorem \ref{t-ipk} follows as an immediate
corollary of this result. We give two proofs of Theorem \ref{t-ipkdescomp}.
Our first proof recovers Theorem \ref{t-ipkdescomp} from two classical
results of H.\ O.\ Foulkes from \cite{Foulkes1976}, a foundational
paper relating symmetric function theory and representation theory
to permutation enumeration. Our second proof is a direct bijective
proof, which is in a sense equivalent to the symmetric function proof
but does not rely on the machinery of symmetric functions. We end
Section 3 with a few other elementary consequences of Theorem \ref{t-ipkdescomp}.
Finally, in Section 4, we give a bijective proof of Theorem \ref{t-ilpk}
using regular expressions and tilings.

\section{Generating function proofs for Theorems \ref{t-ipk} and \ref{t-ilpk}}

Given a consecutive pattern $\sigma$, let $P_{\sigma,n}^{\ipk}(t)$
be defined by 
\[
P_{\sigma,n}^{\ipk}(t)\coloneqq\sum_{\pi\in\mathfrak{S}_{n}(\sigma)}t^{\ipk(\pi)+1}
\]
for $n\geq1$ and $P_{\sigma,0}^{\ipk}(t)\coloneqq1$. (It may seem more natural to define these polynomials without an extra factor of $t$, but the definition given above leads to nicer formulas.) The theorem
below, derived in \cite{Zhuang2021}, gives a generating function
formula for the polynomials $P_{12\cdots m,n}^{\ipk}(t)$.
\begin{thm}[{\cite[Theorem 4.5 (c)]{Zhuang2021}}]
\label{t-ipkgf}Let $m\geq2$. We have 
\begin{multline*}
\qquad\frac{1}{1-t}+\frac{1}{2}\sum_{n=1}^{\infty}\left(\frac{1+t}{1-t}\right)^{n+1}P_{12\cdots m,n}^{\ipk}\left(\frac{4t}{(1+t)^{2}}\right)x^{n}\\
=1+\sum_{k=1}^{\infty}\left[1-2kx+\sum_{j=1}^{\infty}(c_{m,j,k}x^{jm}-c_{m,j,k}^{\prime}x^{jm+1})\right]^{-1}t^{k}\qquad
\end{multline*}
where 
\[
c_{m,j,k}={\displaystyle 2\sum_{l=1}^{k}{l+jm-1 \choose l-1}{jm-1 \choose k-l}}\quad\text{and}\quad c_{m,j,k}^{\prime}=2\sum_{l=1}^{k}{l+jm \choose l-1}{jm \choose k-l}.
\]
\end{thm}

We shall use Theorem \ref{t-ipkgf} to prove Theorem \ref{t-ipk}.
Given two formal power series $f$ and $g$ in the variable $t$,
let us write $f\sim g$ if they have the same linear coefficients,
i.e., $[t]\,f=[t]\,g$.
\begin{proof}[Proof of Theorem \ref{t-ipk}]
Taking Theorem \ref{t-ipkgf}, replacing $x$ with $(1-t)x/(1+t)$,
and performing a few algebraic manipulations gives
\begin{multline*}
\qquad\sum_{n=1}^{\infty}P_{12\cdots m,n}^{\ipk}\left(\frac{4t}{(1+t)^{2}}\right)x^{n}=\frac{2}{1+t}\left[(1-t)\left[1+\sum_{k=1}^{\infty}\left[1-\frac{2kx(1-t)}{1+t}\vphantom{+\sum_{j=1}^{\infty}\left(c_{m,j,k}\left(\frac{x(1-t)}{1+t}\right)^{jm}-c_{m,j,k}^{\prime}\left(\frac{x(1-t)}{1+t}\right)^{jm+1}\right)}\right.^{\vphantom{-1}}\right.\right.\\
\left.\left.\left.+\sum_{j=1}^{\infty}\left(c_{m,j,k}\left(\frac{x(1-t)}{1+t}\right)^{jm}-c_{m,j,k}^{\prime}\left(\frac{x(1-t)}{1+t}\right)^{jm+1}\right)\right]^{-1}t^{k}\right]-1\right].\qquad
\end{multline*}
Let us replace the variable $t$ with $v$, and let $t=4v/(1+v)^{2}$.
It can be readily verified that solving $t=4v/(1+v)^{2}$ yields $v=2t^{-1}(1-\sqrt{1-t})-1$.
Therefore, we have
\begin{multline*}
\qquad\sum_{n=1}^{\infty}P_{12\cdots m,n}^{\ipk}(t)x^{n}=\frac{2}{1+v}\left[(1-v)\left[1+\sum_{k=1}^{\infty}\left[1-\frac{2kx(1-v)}{1+v}\vphantom{+\sum_{j=1}^{\infty}\left(c_{m,j,k}\left(\frac{x(1-v)}{1+v}\right)^{jm}-c_{m,j,k}^{\prime}\left(\frac{x(1-v)}{1+v}\right)^{jm+1}\right)}\right.^{\vphantom{-1}}\right.\right.\\
\left.\left.\left.+\sum_{j=1}^{\infty}\left(c_{m,j,k}\left(\frac{x(1-v)}{1+v}\right)^{jm}-c_{m,j,k}^{\prime}\left(\frac{x(1-v)}{1+v}\right)^{jm+1}\right)\right]^{-1}v^{k}\right]-1\right].\qquad
\end{multline*}
where $v$ is as above. In fact, $v$ is a formal power series in
$t$, and we have 
\[
v=\frac{1}{4}t+\frac{1}{8}t^{2}+\frac{5}{64}t^{3}+\cdots,
\]
\[
\frac{2}{1+v}=2-\frac{1}{2}t-\frac{1}{8}t^{2}-\frac{1}{16}t^{3}+\cdots,
\]
and
\[
\frac{1-v}{1+v}=1-\frac{1}{2}t-\frac{1}{8}t^{2}-\frac{1}{16}t^{3}+\cdots.
\]
As we are only concerned with the linear coefficients, we may truncate
these series to obtain {\allowdisplaybreaks 
\begin{alignat*}{1}
 & \sum_{n=1}^{\infty}P_{12\cdots m,n}^{\ipk}(t)x^{n}\\
 & \quad\sim\left(2-\frac{t}{2}\right)\left[\left(1-\frac{t}{4}\right)\left[1+\sum_{k=1}^{\infty}\left[1-2kx+\sum_{j=1}^{\infty}\left(c_{m,j,k}x^{jm}-c_{m,j,k}^{\prime}x^{jm+1}\right)\right]^{-1}\left(\frac{t}{4}\right)^{k}\right]-1\right]\\
 & \quad\sim-t+2\left[1-2x+\sum_{j=1}^{\infty}\left(c_{m,j,1}x^{jm}-c_{m,j,1}^{\prime}x^{jm+1}\right)\right]^{-1}\frac{t}{4}+\frac{t}{2}.
\end{alignat*}
Some additional algebraic manipulations show that 
\begin{align*}
\sum_{n=1}^{\infty}P_{12\cdots m,n}^{\ipk}(t)x^{n} & \sim-t+2\left[1-2x+\sum_{j=1}^{\infty}\left(c_{m,j,1}x^{jm}-c_{m,j,1}^{\prime}x^{jm+1}\right)\right]^{-1}\frac{t}{4}+\frac{t}{2}\\
 & =\frac{t}{2}\left[1-2x+\sum_{j=1}^{\infty}(2x^{jm}-2x^{jm+1})\right]^{-1}-\frac{t}{2}\\
 & =\frac{t}{2}\cdot\frac{1}{1-2x+2(1-x)\sum_{j=1}^{\infty}x^{jm}}-\frac{t}{2}\\
 & =\frac{t}{2}\cdot\frac{1}{1-2x+2(1-x)\left(\frac{1}{1-x^{m}}-1\right)}-\frac{t}{2}\\
 & =\frac{t(x-x^{m})}{1-2x+x^{m}}\\
 & =t\left(\frac{1-x}{1-2x+x^{m}}-1\right).
\end{align*}
Because $(1-x)/(1-2x+x^{m})$ is the ordinary generating function
for the Fibonacci numbers of order $m-1$ \cite[A048887]{oeis}, the
result follows.}
\end{proof}
Let us now turn our attention to Theorem \ref{t-ilpk}. Given a consecutive
pattern $\sigma$, define 
\[
P_{\sigma,n}^{\ilpk}(t)\coloneqq\sum_{\pi\in\mathfrak{S}_{n}(\sigma)}t^{\ilpk(\pi)}
\]
for all $n\geq0$. Then the following is an analogue of Theorem \ref{t-ipkgf}
for the polynomials $P_{m\cdots21,n}^{\ilpk}(t)$.
\begin{thm}[{\cite[Theorem 4.8 (c)]{Zhuang2021}}]
\label{t-ilpkgf}Let $m\geq2$. We have 
\begin{multline*}
\qquad\sum_{n=0}^{\infty}\frac{(1+t)^{n}}{(1-t)^{n+1}}P_{m\cdots21,n}^{\ilpk}\left(\frac{4t}{(1+t)^{2}}\right)x^{n}\\
=\frac{1}{1-x}+\sum_{k=1}^{\infty}\left[1-(2k+1)x+\sum_{j=1}^{\infty}(e_{m,j,k}x^{jm}-e_{m,j,k}^{\prime}x^{jm+1})\right]^{-1}t^{k}\qquad
\end{multline*}
where 
\[
e_{m,j,k}=4\sum_{l=1}^{k}{l+jm-1 \choose l-1}{jm-2 \choose k-l}\quad\text{and}\quad e_{m,j,k}^{\prime}=4\sum_{l=1}^{k}{l+jm \choose l-1}{jm-1 \choose k-l}.
\]
\end{thm}

\begin{proof}[Proof of Theorem \ref{t-ilpk}]
Following the proof of Theorem \ref{t-ipk} given above, we obtain
from Theorem \ref{t-ilpkgf} the formula
\begin{multline*}
\qquad\sum_{n=0}^{\infty}P_{m\cdots21,n}^{\ilpk}(t)x^{n}=(1-v)\left[\frac{1}{1-\frac{x(1-v)}{1+v}}+\sum_{k=1}^{\infty}\left[1-\frac{(2k+1)x(1-v)}{1+v}\vphantom{+\sum_{j=1}^{\infty}\left(e_{m,j,k}\left(\frac{x(1-v)}{1+v}\right)^{jm}-e_{m,j,k}^{\prime}\left(\frac{x(1-v)}{1+v}\right)^{jm+1}\right)}\right.^{\vphantom{-1}}\right.\\
\left.\left.+\sum_{j=1}^{\infty}\left(e_{m,j,k}\left(\frac{x(1-v)}{1+v}\right)^{jm}-e_{m,j,k}^{\prime}\left(\frac{x(1-v)}{1+v}\right)^{jm+1}\right)\right]^{-1}v^{k}\right]\qquad
\end{multline*}
where again $v=2t^{-1}(1-\sqrt{1-t})-1$. Because 
\[
v=\frac{1}{4}t+\frac{1}{8}t^{2}+\frac{5}{64}t^{3}+\cdots,
\]
\[
\frac{1-v}{1+v}=1-\frac{1}{2}t-\frac{1}{8}t^{2}-\frac{1}{16}t^{3}+\cdots,
\]
and 
\[
\frac{1}{1-\frac{x(1-v)}{1+v}}=\frac{1}{1-x}-\frac{x}{2(1-x)^{2}}t+\frac{x(3x-1)}{8(1-x)^{3}}t^{2}-\frac{x(5x^{2}-4x+1)}{16(1-x)^{4}}t^{3}+\cdots,
\]
we have {\allowdisplaybreaks 
\begin{alignat*}{1}
 & \sum_{n=0}^{\infty}P_{m\cdots21,n}^{\ilpk}(t)x^{n}\sim\left(1-\frac{t}{4}\right)\left[\frac{1}{1-x}-\frac{xt}{2(1-x)^{2}}+\sum_{k=1}^{\infty}\left[1-(2k+1)x\vphantom{+\sum_{j=1}^{\infty}(e_{m,j,k}x^{jm}-e_{m,j,k}^{\prime}x^{jm+1})}\right.^{\vphantom{-1}}\right.\\
 & \qquad\qquad\qquad\qquad\qquad\qquad\qquad\qquad\qquad\left.\left.+\sum_{j=1}^{\infty}(e_{m,j,k}x^{jm}-e_{m,j,k}^{\prime}x^{jm+1})\right]^{-1}\left(\frac{t}{4}\right)^{k}\right]\\
 & \qquad\sim\left(1-\frac{t}{4}\right)\left[\frac{1}{1-x}-\frac{xt}{2(1-x)^{2}}+\left[1-3x+\sum_{j=1}^{\infty}(e_{m,j,1}x^{jm}-e_{m,j,1}^{\prime}x^{jm+1})\right]^{-1}\frac{t}{4}\right]\\
 & \qquad\sim-\frac{t}{4(1-x)}-\frac{xt}{2(1-x)^{2}}+\left[1-3x+\sum_{j=1}^{\infty}(e_{m,j,1}x^{jm}-e_{m,j,1}^{\prime}x^{jm+1})\right]^{-1}\frac{t}{4}\\
 & \qquad=-\frac{t}{4(1-x)}-\frac{xt}{2(1-x)^{2}}+\left[1-3x+\sum_{j=1}^{\infty}(4x^{jm}-4x^{jm+1})\right]^{-1}\frac{t}{4}\\
 & \qquad=-\frac{t}{4(1-x)}-\frac{xt}{2(1-x)^{2}}+\frac{1}{1-3x+4(1-x)\sum_{j=1}^{\infty}x^{jm}}\cdot\frac{t}{4}\\
 & \qquad=-\frac{t}{4(1-x)}-\frac{xt}{2(1-x)^{2}}+\frac{1}{1-3x+4(1-x)\left(\frac{1}{1-x^{m}}-1\right)}\cdot\frac{t}{4}\\
 & \qquad=\frac{x^{2}(x^{m-2}-1)t}{(1-x)^{2}(x^{m+1}-3x^{m}+3x-1)}.
\end{alignat*}
For $m=3$, this specializes to 
\begin{align*}
\sum_{n=0}^{\infty}P_{321,n}^{\ilpk}(t)x^{n} & \sim\frac{x^{2}(x-1)t}{(1-x)^{2}(x^{4}-3x^{3}+3x-1)}\\
 & =\frac{x^{2}t}{(1-x)^{2}(1-2x-2x^{2}+x^{3})},
\end{align*}
and since $x^{2}/((1-x)^{2}(1-2x-2x^{2}+x^{3}))$ is the ordinary
generating function for \cite[A080145]{oeis} (shifted by 1), the
result follows.}
\end{proof}

\section{Descent compositions and inverse peaks: alternative proofs of Theorem \ref{t-ipk}}

Every permutation can be decomposed into a sequence of \textit{increasing
runs}\textemdash maximal increasing consecutive subsequences\textemdash and
the \textit{descent composition} of a permutation $\pi$ is the integer
composition whose parts are the lengths of the increasing runs of
$\pi$ in the order that they appear. For example, the increasing
runs of $\pi=85712643$ are $8$, $57$, $126$, $4$, and $3$, and
the descent composition of $\pi$ is $(1,2,3,1,1)$. It is easy to
see that a permutation avoids the consecutive pattern $12\cdots m$
if and only if its descent composition has every part less than $m$.

The following theorem is the main result of this section.
\begin{thm}
\label{t-ipkdescomp} For any composition $L$ of $n\geq1$, there
exists exactly one permutation $\pi\in\mathfrak{S}_{n}$ with descent
composition $L$ such that $\ipk(\pi)=0$.
\end{thm}

By restricting Theorem \ref{t-ipkdescomp} to compositions with all
parts less than $m$, we see that permutations $\pi$ in $\mathfrak{S}_{n}(12\cdots m)$
with $\ipk(\pi)=0$ are in one-to-one correspondence with compositions
of $n$ with all parts less than $m$. Hence Theorem \ref{t-ipk}
is an immediate corollary of Theorem \ref{t-ipkdescomp}.

In this section we give two (related) proofs for Theorem \ref{t-ipkdescomp},
thus resulting in two additional proofs of Theorem \ref{t-ipk}. For
the first proof, we will assume familiarity with some basic definitions
from the theory of symmetric functions at the level of Stanley \cite[Chapter 7]{Stanley2001},
but let us briefly establish notation and review a few elementary
facts which will be needed for the proof.

Given a partition $\lambda$, let $s_{\lambda}$ denote the Schur
function of shape $\lambda$, and if $\mu$ is a partition contained
inside $\lambda$, then let $s_{\lambda/\mu}$ denote the skew Schur
function of shape $\lambda/\mu$. Recall that a connected skew shape
with no $2\times2$ square is called a \textit{ribbon shape}.\footnote{Ribbon shapes are also commonly called ``skew-hooks'' (e.g., by Foulkes
\cite{Foulkes1976}) or ``border strips'' (e.g., by Stanley \cite{Stanley2001}).} Given a composition $L=(L_{1},L_{2},\dots,L_{k})$, let $r_{L}$
denote the skew Schur function of the ribbon shape with $L_{i}$ squares
in row $k-i+1$ for each $i$. See Figure 1 for an example.

\begin{figure}
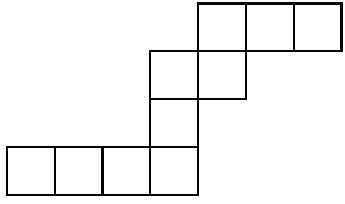

\noindent \begin{centering}
\ydiagram{4+3,3+2,3+1,4}\smallskip{}
\par\end{centering}
\caption{The ribbon shape corresponding to the composition $(4,1,2,3)$.}
\end{figure}

Let $\left\langle \cdot,\cdot\right\rangle $ denote the usual scalar
product on symmetric functions, and let $c_{\mu,\nu}^{\lambda}=\left\langle s_{\lambda/\mu},s_{\nu}\right\rangle $.
Recall that the $c_{\mu,\nu}^{\lambda}$ are called \textit{Littlewood\textendash Richardson
coefficients} and are the structure constants for products of Schur
functions when expanded back in the Schur basis; that is,
\[
s_{\mu}s_{\nu}=\sum_{\lambda}c_{\mu,\nu}^{\lambda}s_{\lambda}.
\]

We next state a couple lemmas from which Theorem \ref{t-ipkdescomp}
follows. The first lemma is equivalent to Theorem 6.2 of Foulkes \cite{Foulkes1976},
although it was first stated in the form below by Gessel \cite[Theorem 5]{Gessel1984};
see also \cite[Corollary 7.23.8]{Stanley2001}.
\begin{lem}
\label{l-scalprod}Let $L$ and $M$ be compositions of $n\geq1$.
Then $\left\langle r_{L},r_{M}\right\rangle $ is the number of permutations
$\pi$ with descent composition $L$ such that $\pi^{-1}$ has descent
composition $M$.
\end{lem}

The second lemma is Theorem 7.1 of Foulkes \cite{Foulkes1976}.
\begin{lem}
\label{l-LRcoeff}Let $\lambda$, $\mu$, and $\nu$ be partitions
of $n\geq1$ such that $\lambda$ has $p$ columns and $\nu=(n-r,1^{r})$
for some $0\leq r\leq n$. Then
\[
c_{\mu,\nu}^{\lambda}=\delta_{n-r,p}=\begin{cases}
1, & \text{if }n-r=p,\\
0, & \text{otherwise.}
\end{cases}
\]
\end{lem}

\begin{proof}[First proof of Theorem \ref{t-ipkdescomp}]
Let $r_{L}$ be the ribbon Schur function corresponding to the composition
$L$. Furthermore, let us write $r_{L}=s_{\lambda/\mu}$ and let $p$
be the number of columns of $\lambda$. It is easy to see that a permutation
in $\mathfrak{S}_{n}$ has no peaks if and only if its descent composition
is of the form $(1^{r},n-r)$ for some $0\leq r\leq n$. Also, the
ribbon shape corresponding to the composition $(1^{r},n-r)$ is precisely
the Young diagram corresponding to the partition $(n-r,1^{r})$, so
that $r_{(1^{r},n-r)}=s_{(n-r,1^{r})}$. Thus, in light of Lemma \ref{l-scalprod},
the number of permutations $\pi\in\mathfrak{S}_{n}$ with descent
composition $L$ such that $\pi^{-1}$ has no peaks is equal to
\begin{align*}
\sum_{r=0}^{n}\left\langle r_{L},r_{(1^{r},n-r)}\right\rangle  & =\sum_{r=0}^{n}\left\langle s_{\lambda/\mu},s_{(n-r,1^{r})}\right\rangle \\
 & =\sum_{r=0}^{n}c_{\mu,(n-r,1^{r})}^{\lambda} & \text{(since \ensuremath{c_{\mu,\nu}^{\lambda}}=\ensuremath{\left\langle s_{\lambda/\mu},s_{\nu}\right\rangle \text{)}}}\\
 & =\sum_{r=0}^{n}\delta_{n-r,p} & \text{(by Lemma }\ref{l-LRcoeff}\text{)}\\
 & =1.
\end{align*}
Hence the result follows.
\end{proof}
Now, we provide a bijective proof of Theorem \ref{t-ipkdescomp}.
\begin{proof}[Second proof of Theorem \ref{t-ipkdescomp}]
Let $L=(L_{1},L_{2},\dots,L_{k})$ be a composition of $n$. We construct
a permutation $\pi$ with descent composition $L$ and whose inverse
has no peaks by inserting the letters $1,2,\dots,n$ in a line with
$n$ positions in the following way:
\begin{enumerate}
\item Let 1 be the first letter of the $k$th increasing run, let 2 be the
first letter of the $(k-1)$th increasing run, and so on. This step
determines the positions of the letters $1,2,\dots,k$.
\item Insert the remaining letters $k+1,k+2,\dots,n$ into the remaining
positions in ascending order.
\end{enumerate}
For example, suppose that $L=(3,2,3,1)$. Then the resulting permutation
would be $\pi=456372891$, which indeed has descent composition $L$
and whose inverse $\pi^{-1}=964123578$ has no peaks. In general,
by placing the smallest $k$ letters in the initial positions of each
increasing run and the remaining letters in ascending order, it is
clear that the resulting permutation must have descent composition
$L$. Furthermore, observe that $\pi^{-1}$ consists of a decreasing
run formed by the positions of the smallest $k$ letters in $\pi$,
followed by an increasing run formed by the positions of the remaining
letters in $\pi$. Thus $\pi^{-1}$ has no peaks.

We now argue that any permutation whose inverse has no peaks must
have the form given above. Let $\pi$ be such a permutation in $\mathfrak{S}_{n}$
with $\pi^{-1}(k)=1$. Since $\pi^{-1}$ has no peaks, the first $k$
letters of $\pi^{-1}$ are descending and the remaining letters of
$\pi^{-1}$ are ascending. This means that $\pi$ is a shuffle of
the decreasing subsequence $k\cdots21$ (with $k$ in the first position
of $\pi$) and the increasing subsequence $(k+1)(k+2)\cdots n$. Hence,
$\pi$ is of the form given above, which completes the proof.
\end{proof}

There is a simple idea underlying the above proof: the permutations whose inverse has no peaks
are precisely the permutations whose entries can be drawn on a ``$<$'' shape. 
This set of permutations is an example of a \emph{monotone grid class}, a type of permutation class arising in the
study of permutation patterns \cite{Bevan,HuczynskaVatter,MurphyVatter}.

Also, we note that the bijection used in the above proof is closely related
to a bijection between ribbon shapes and standard Young tableaux of
hook shape which Foulkes uses to prove Lemma \ref{l-LRcoeff}. Thus,
our second proof of Theorem \ref{t-ipkdescomp} is in a sense equivalent
to our first proof, but described without using the language of symmetric
functions.

Finally, we give a few other elementary consequences of Theorem \ref{t-ipkdescomp}.
Recall that a permutation $\pi$ in $\mathfrak{S}_{n}$
is \textit{alternating} if $\pi_{1}<\pi_{2}>\pi_{3}<\cdots$, and
is \textit{reverse-alternating} if $\pi_{1}>\pi_{2}<\pi_{3}>\cdots$.
\begin{cor}
For any $n\geq1$, there is exactly one alternating permutation $\pi$
in $\mathfrak{S}_{n}$ with $\ipk(\pi)=0$, and exactly one reverse-alternating
permutation $\pi$ in $\mathfrak{S}_{n}$ with $\ipk(\pi)=0$.
\end{cor}

\begin{proof}
It is easy to see that a permutation in $\mathfrak{S}_{n}$ is alternating
if and only if its descent composition is of the form $(2,2\dots,2)$
or $(2,2,\dots,2,1)$, depending on the parity of $n$. Similarly,
a permutation in $\mathfrak{S}_{n}$ is reverse-alternating if and
only if its descent composition is of the form $(1,2,2\dots,2,1)$
or $(1,2,2,\dots,2)$. Therefore, the result follows from Theorem
\ref{t-ipkdescomp}.
\end{proof}
Given a permutation $\pi$ in $\mathfrak{S}_{n}$,
we say that $i\in[n-1]$ is a \textit{descent} of $\pi$ if $\pi_{i}>\pi_{i+1}$.
Let $\des(\pi)$ denote the number of descents of $\pi$, which is
one less than the number of parts of the descent composition of $\pi$.
Also, let $\pk(\pi)$ be the number of peaks of $\pi$ and $\lpk(\pi)$
the number of left peaks of $\pi$.
\begin{cor}
For any $n\geq1$ and $k\geq0$, the number of permutations $\pi$
in $\mathfrak{S}_{n}$ with $\des(\pi)=k$ and $\ipk(\pi)=0$ is equal
to ${n-1 \choose k}$.
\end{cor}

\begin{proof}
This is an immediate consequence of Theorem \ref{t-ipkdescomp} and
the well-known fact that, for any integers $n,k\geq1$, there are
${n-1 \choose k-1}$ compositions of $n$ into $k$ parts \cite[A097805]{oeis}.
\end{proof}
\begin{cor}
For any $n\geq1$ and $k\geq0$, the number of permutations $\pi$
in $\mathfrak{S}_{n}$ with $\pk(\pi)=k$ and $\ipk(\pi)=0$ is equal
to ${n \choose 2k+1}$.
\end{cor}

\begin{proof}
The number of peaks of a permutation $\pi$ is equal to the number
of increasing runs of $\pi$ of length greater than 1, not including
the final increasing run \cite[Lemma 2.1 (a)]{Gessel2018}. Furthermore,
it is known that there are ${n \choose 2k}$ compositions of $n$
with $k$ parts greater than 1 \cite[A034839]{oeis}. Therefore, by
conditioning on the number of letters of $\pi$ prior to the final
increasing run and using Theorem \ref{t-ipkdescomp}, we see that
there are $\sum_{j=0}^{n-1}{j \choose 2k}$ permutations $\pi$ in $\mathfrak{S}_{n}$
with $\pk(\pi)=k$ and $\ipk(\pi)=0$. The result then follows from
the identity $\sum_{j=0}^{n-1}{j \choose 2k}={n \choose 2k+1}$ \cite[A034867]{oeis}.
\end{proof}
\begin{cor}
For any $n\geq1$ and $k\geq0$, the number of permutations $\pi$
in $\mathfrak{S}_{n}$ with $\lpk(\pi)=k$ and $\ipk(\pi)=0$ is equal
to ${n \choose 2k}$.
\end{cor}

\begin{proof}
Let us call $i\in\{2,3,\dots,n\}$ a \textit{right peak} of $\pi$
if either $i$ is a peak of $\pi$ or if $i=n$ and $\pi_{n}>\pi_{n-1}$.
Let $\rpk(\pi)$ denote the number of right peaks of $\pi$. Define
$C_{n,k}^{\lpk}$ to be the set of compositions which are descent
compositions of permutations in $\mathfrak{S}_{n}$ with $k$ left
peaks, and define $C_{n,k}^{\rpk}$ in the analogous way but for right
peaks. It is known that $\rpk(\pi)$ is equal to the number of increasing
runs of $\pi$ of length greater than 1 \cite[Lemma 2.1 (d)]{Gessel2018}.
Again using the fact that there are ${n \choose 2k}$ compositions
of $n$ with $k$ parts greater than 1, we see that $|C_{n,k}^{\rpk}|={n \choose 2k}$.
In light of Theorem \ref{t-ipkdescomp}, it suffices to show that
$|C_{n,k}^{\lpk}|={n \choose 2k}$; we do this by giving a bijection
between $C_{n,k}^{\lpk}$ and $C_{n,k}^{\rpk}$.

Given a permutation $\pi=\pi_{1}\pi_{2}\cdots\pi_{n}$ in $\mathfrak{S}_{n}$,
define the \textit{reverse} of $\pi$ to be $\pi^{r}\coloneqq\pi_{n}\cdots\pi_{2}\pi_{1}$.
Clearly, the reversal operation $\pi\mapsto\pi^{r}$ is an involution
on $\mathfrak{S}_{n}$ which toggles between left peaks and right
peaks. Moreover, reversal induces an involution on compositions: if
$L$ is a composition, then define $L^{r}$ to be the descent composition
of $\pi^{r}$, where $\pi$ is any permutation whose descent composition
is $L$. (It is easy to see that $L^{r}$ does not depend on the specific
choice of $\pi$.) This involution on compositions restricts to a
bijection between $C_{n,k}^{\lpk}$ and $C_{n,k}^{\rpk}$, and we
are done.
\end{proof}

\section{A bijective proof of Theorem \ref{t-ilpk}}

In this final section, we give a bijective proof of Theorem \ref{t-ilpk}. Here, we assume basic familiarity with regular expressions. For an introduction to regular expressions as they arise in the theory of computation, see \cite[Sec.\ 1.3]{Sipser}; for an explication of the role of regular expressions in symbolic combinatorics, see \cite[Secs.\ I.4 \& A.7]{Flajolet2009}.

Our proof consists of multiple stages, and before we begin the proof, let us establish some new notation and outline the steps of our proof. Define the sets $N_n$ and $N_{n}^{\prime}$ by
\[
N_{n}\coloneqq\{\,\pi\in\mathfrak{S}_{n} : \lpk(\pi)=1\,\}\quad\text{and}\quad N_{n}^{\prime}\coloneqq\{\,\pi\in N_{n} : \pi^{-1}\in\mathfrak{S}_{n}(321)\,\}.
\]
It is clear that $N_{n}^{\prime}$ is in bijection with permutations in $\mathfrak{S}_n(321)$ satisfying $\ilpk(\pi)=1$, which are precisely the permutations counted by Theorem \ref{t-ilpk}.

Let $A^n$ denote the set of words of length $n$ on the alphabet $\{a,b,c\}$, and let $W_n$ denote the set of words $w \in A^n$ of the form $w = a^i c u ac^j$, where $i,j\geq 0$ and $u\in A^{n-i-j-2}$, such that $w$ avoids the subwords $bba$, $bbb$, $cba$, and $cbb$. (We say that $w$ \emph{avoids} $v$ if $w$ does not contain $v$ as a subword, i.e., we cannot write $u = \alpha v \beta$ for some---possibly empty---words $\alpha$ and $\beta$.) The first step of our proof will be to establish the following proposition.

\begin{prop} \label{lem:NW}
The sets $N_{n}^{\prime}$ and $W_n$ are in bijection for every $n \geq 1$.
\end{prop}

Upon proving Proposition \ref{lem:NW}, it remains to show $|W_n| = \sum_{i=1}^{n-1} \sum_{k=1}^i f_{k-1} f_k$. It is easily checked that $\sum_{i=1}^{n-1} \sum_{k=1}^i f_{k-1} f_k = \sum_{k=1}^{n-1} \sum_{j=0}^{n-k-1} f_{k-1} f_k$, so we will instead show that 
\begin{equation} \label{eqn:fibonaccisum}
|W_n| = \sum_{k=1}^{n-1} \sum_{j=0}^{n-k-1} f_{k-1} f_k.
\end{equation}
To do so, we partition $W_n$ into subsets $W_n^{j,k}$ corresponding to the summands in Equation (\ref{eqn:fibonaccisum}), and then construct a bijection between each $W_n^{j,k}$ and a certain set of tilings whose cardinality is $f_{k-1} f_k$. To define the subsets $W_n^{j,k}$, we shall need to characterize the words in $W_n$ as those matching a certain regular expression. Let $\theta^*$ denote the Kleene star of a regular expression $\theta$, and let $\theta^+ \coloneqq \theta^* \theta$.

\begin{prop} \label{lem:regexmain}
Let $n \geq 1$ and $w \in A^n$. Then we have $w \in W_n$ if and only if $w$ matches the regular expression $a^* c \,(c \cup bc \cup a^+ b \cup a^+ c)^*\, a^+ c^*$.
\end{prop}

Note that $a^* c \,(c \cup bc \cup a^+ b \cup a^+ c)^*\, a^+ c^*$ is an unambiguous regular expression, in the sense that every word that matches it does so in a unique way (see \cite[Sec.\ A.7]{Flajolet2009}). We could then ``translate'' this regular expression into a generating function for the words that match it. This generating function turns out to be $x^{2}/((1-x)^{2}(1-2x-2x^{2}+x^{3}))$, which appears in our earlier generating function proof of Theorem 2. Our present focus, however, is on bijective proof.

Let $Z_k$ be the set of words in $A^k$ that match the regular expression $a^* c \,(c \cup bc \cup a^+ b \cup a^+ c)^*$. If $w$ matches the regular expression $a^* c \,(c \cup bc \cup a^+ b \cup a^+ c)^*\, a^+ c^*$, then we can write $w$ uniquely in the form $w = za^{n-j-k} c^j$, where $1 \le k \le n-1$, $0 \le j \le n-k-1$, and $z \in Z_k$. Thus we can define
\[ W_n^{j,k} \coloneqq \{\, za^{n-j-k} c^j \in A^n : z \in Z_k \, \}. \]

Observe that each $W_n^{j,k}$ is a subset of $W_n$ by Proposition \ref{lem:regexmain}. Hence, for each $n\geq 1$, the sets $W_n^{j,k}$ (ranging over $1 \le k \le n-1$ and $0 \le j \le n-k-1$) form a partition of $W_n$. Moreover, there is a bijection from $W_n^{j,k}$ to $Z_k$: namely, remove the $a^{n-j-k} c^j$ suffix. Therefore, we have
\[ |W_n| = \sum_{k=1}^{n-1} \sum_{j=0}^{n-k-1} |W_n^{j,k}| = \sum_{k=1}^{n-1} \sum_{j=0}^{n-k-1} |Z_k| \]
and the last remaining step will be to prove the following proposition.

\begin{prop} \label{lem:Zfib}
We have $|Z_k| = f_{k-1} f_k$ for every $k\geq1$.
\end{prop}

Now that we have outlined our bijective proof of Theorem \ref{t-ilpk}, we will now fill in the details by proving Propositions \ref{lem:NW}--\ref{lem:Zfib}.

\subsection{Proof of Proposition \ref{lem:NW}}

Recall that $N_n$ is the set of permutations in $\S_n$ with exactly one left peak. Given $\pi \in \S_n$, observe that $\pi \in N_n$ if and only if (i) $\pi \not= 12\cdots n$ and (ii) $\pi$ can be written in the form $\pi = \alpha \beta \gamma$, where $\alpha$, $\beta$, and $\gamma$ are subwords of $\pi$ (possibly of length $0$) such that $\alpha$ is increasing, $\beta$ is decreasing, and $\gamma$ is increasing.

For example, if $\pi = 1\,2\,5\,10\,12\,8\,6\,4\,3\,7\,9\,11 \in \S_{12}$, then $\pi \in N_{12}$, and we can take $\alpha = 1\,2\,5\,10$ and $\beta = 12\,8\,6\,4\,3$ and $\gamma = 7\,9\,11$. See Figure \ref{fig:Nperm}; the reason we have named this set $N_n$ is because the permutations are shaped like the letter \textit{N}\@. 
Like the set of permutations whose inverse has no peaks, the set $\bigsqcup_{n\ge0} (N_n \cup \{12\cdots n\})$ is an example of a monotone 
grid class (see the remark after the second proof of Theorem \ref{t-ipkdescomp} in Section 3).


We can also characterize $N_n$ using valleys. Given $\pi$ in $\mathfrak{S}_{n}$, we say that $i\in\{2,3,\dots,n-1\}$ is a \textit{valley} of $\pi$ if $\pi_{i-1}>\pi_{i}<\pi_{i+1}$, and we say that $i\in\{2,3,\dots,n\}$ is a \textit{right valley} of $\pi$ if either $i$ is a valley of $\pi$ or if $i=n$ and $\pi_{n-1}>\pi_n$. Then $N_n$ is the set of permutations in $\S_n$ with exactly one right valley. Note that if $\pi \in N_n$, then the left peak of $\pi$ must be less than the right valley of $\pi$.


\begin{figure}
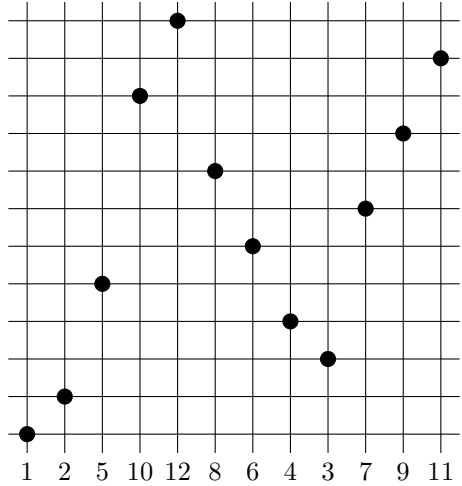

\[ \drawpermutation{1,2,5,10,12,8,6,4,3,7,9,11}{12} \]
\caption{The permutation $\pi = 1\,2\,5\,10\,12\,8\,6\,4\,3\,7\,9\,11 \in \S_{12}$ can be decomposed into an increasing sequence $\alpha = 1\,2\,5\,10$ followed by a decreasing sequence $\beta = 12\,8\,6\,4\,3$ followed by an increasing sequence $\gamma = 7\,9\,11$. The permutation is shaped like an \textit{N}. \label{fig:Nperm}}
\end{figure}

The decomposition $\alpha \beta \gamma$ is not unique. Given $\pi \in N_n$, let $i$ be the left peak of $\pi$ and $j$ the right valley of $\pi$. In the decomposition $\pi = \alpha \beta \gamma$, the letters $\pi_1, \pi_2, \ldots \pi_{i-1}$ must be in $\alpha$; the letters $\pi_{i+1}, \ldots, \pi_{j-1}$ must be in $\beta$; and the letters $\pi_{j+1}, \ldots, \pi_{n}$ must be in $\gamma$. The only choices are to place $\pi_{i}$ in either $\alpha$ or $\beta$ and to place $\pi_{j}$ in either $\beta$ or $\gamma$. Thus there are $2 \cdot 2 = 4$ possible decompositions $\pi = \alpha\beta\gamma$. Of these four, we define the \emph{canonical decomposition} of $\pi$ to be the one which places $\pi_{i}$ in $\alpha$ and places $\pi_{j}$ in $\gamma$.

Continuing the example above with $\pi = 1\,2\,5\,10\,12\,8\,6\,4\,3\,7\,9\,11$, we have $\pi_i = 12$ and $\pi_j = 3$, and the canonical decomposition is $\alpha\,|\,\beta\,|\,\gamma = 1\,2\,5\,10\,12\,|\,8\,6\,4\,|\,3\,7\,9\,11$. (We shall often write $\alpha\beta\gamma$ as $\alpha\,|\,\beta\,|\,\gamma$ to clearly demarcate the blocks of the decomposition.) Observe that the canonical decomposition of a permutation is the one with the shortest possible $\beta$. From this point forward, we may refer to the canonical decomposition simply as ``the decomposition''.

Recall that $A^n$ is the set of words of length $n$ on the alphabet $\{a,b,c\}$. A permutation $\pi \in N_n$ determines a word $\phi(\pi) = w_1 w_2 \cdots w_n \in A^n$ in the following way. Let $\pi = \alpha\beta\gamma$ be the decomposition of $\pi$. Then, for each $i \in [n]$, set $w_i = a$ if the value $i$ is in $\alpha$, set $w_i = b$ if $i$ is in $\beta$, and set $w_i = c$ if $i$ is in $\gamma$. This map $\phi \colon N_n \to A^n$ is injective, because each $w_i$ tells us which block $i$ is placed in, which uniquely determines $\pi \in N_n$.

Continuing the example above with $\pi = \alpha\,|\,\beta\,|\,\gamma = 1\,2\,5\,10\,12\,|\,8\,6\,4\,|\,3\,7\,9\,11$, we have that $w_1 = a$ since $1$ is found in $\alpha$, that $w_2 = a$ since $2$ is found in $\alpha$, that $w_3 = c$ since $3$ is found in $\gamma$, and so on. Therefore, $\phi(\pi) = aacbabcbcaca$.

\begin{lem} \label{lem:canonicalword}
Let $w \in A^n$. We have $w \in \phi(N_n)$ if and only if $w$ has the form $w = a^i c u ac^j$, where $i,j \ge 0$ and $u\in A^{n-i-j-2}$.
\end{lem}

\begin{proof}
Assume $w \in \phi(N_n)$, so $w = \phi(\pi)$ for some $\pi \in N_n$. Consider the decomposition $\pi = \alpha\beta\gamma$. If $1$ is in $\alpha$, then let $i$ be the largest number such that $1, 2, \dots, i$ are all in $\alpha$; otherwise set $i=0$. This implies that $w_1 w_2 \cdots w_i = a^i$ and that $w_{i+1}$ is either $b$ or $c$. Since $i+1$ is the smallest letter in $\beta$ or $\gamma$, it must be the letter located at the right valley of $\pi$, which means that it is in $\gamma$; thus, $w_{i+1} = c$. This proves that $w$ begins with $a^i c$.

If $n$ is in $\gamma$, then let $j$ be the largest number such that the values $n-j+1, \ldots, n$ are all found in $\gamma$; otherwise, let $j=0$. A similar argument as above shows that $w$ ends with $ac^j$. Therefore, either $w = a^i c u ac^j$ for some word $u$, or $w = a^i c^j$. But if $w = a^i c^j$ then $\pi = 12\cdots n$, contradicting the definition of $N_n$. Therefore, we have $w = a^i c u ac^j$ for some word $u$.

Conversely, assume that $w$ is in the form $w=a^i c u ac^j$, with $i$, $j$, and $u$ as in the statement of the lemma. Define $\pi \in \S_n$ as follows: For each $p \in [n]$, place $p$ into $\alpha$ if $w_p = a$, place $p$ into $\beta$ if $w_p = b$, and place $p$ into $\gamma$ if $w_p = c$; now sort $\alpha$ in increasing order, sort $\beta$ in decreasing order, sort $\gamma$ in increasing order, and set $\pi = \alpha\beta\gamma$. We know that $i+1$ is placed in $\gamma$ and $n-j$ is placed in $\alpha$ because $w_{i+1}=c$ and $w_{n-j}=a$; thus $\pi \neq 12\cdots n$, and so $\pi \in N_n$.

We will now show that $\alpha\beta\gamma$ is the canonical decomposition of $\pi$; from this, and from the way we defined $\pi$, it immediately follows that $\phi(\pi) = w$. Since $w_1 w_2 \cdots w_{i+1} = a^i c$, the values $1, 2, \ldots, i$ are all in $\alpha$; thus $i+1$ is the lowest value in $\beta$ or $\gamma$, so it must be the letter located at the right valley of $\pi$. Therefore, the letter located at the right valley of $\pi$ is in $\gamma$. A similar argument with $w_{n-j+1} w_{n-j+2} \cdots w_n = ac^j$ shows that the letter located at the left peak of $\pi$ is in $\alpha$. Therefore, $\alpha\beta\gamma$ is the canonical decomposition of $\pi$.
\end{proof}

Recall that $N_n'$ is the subset of $N_n$ containing those permutations whose inverse avoids the consecutive pattern $321$. Equivalently, $N_n'$ is the set of permutations $\pi \in N_n$ such that, for all $i \in [n-2]$, the letters $i$, $i+1$, and $i+2$ do not occur in decreasing order in $\pi$---that is, if $i+1$ occurs to the left of $i$, then $i+2$ must not occur to the left of $i+1$.

\begin{lem} \label{lem:avoidsubword}
Let $\pi \in \S_n$. We have $\pi \in N_n'$ if and only if $\phi(\pi)$ avoids the subwords $bba$, $bbb$, $cba$, and $cbb$.
\end{lem}
\begin{proof}
Set $w = w_1 w_2 \cdots w_n = \phi(\pi)$. Observe that $i$ is an inverse descent of $\pi$---that is, a descent of $\pi^{-1}$---if and only if $i+1$ appears to the left of $i$ in $\pi$. For a given $i \in [n-1]$, we characterize the words $w\in W_n$ such that $i$ is an inverse descent of $\phi(w)$:
\begin{itemize}
\item If $w_i w_{i+1} = aa$, then both $i$ and $i+1$ are in the increasing block $\alpha$ of $\pi$, so $i+1$ appears to the right of $i$. Thus $i$ \emph{is not} an inverse descent of $\pi$.
\item If $w_i w_{i+1} = cc$, then both $i$ and $i+1$ are in the increasing block $\gamma$, so $i$ \emph{is not} an inverse descent of $\pi$ for the same reason as above.
\item If $w_i w_{i+1} = bb$, then both $i$ and $i+1$ are in the decreasing block $\beta$, so $i+1$ appears to the left of $i$. Thus $i$ \emph{is} an inverse descent of $\pi$.
\item If $w_i w_{i+1} \in \{ab, ac, bc\}$, then the block ($\alpha$, $\beta$, or $\gamma$) that contains $i$ is to the left of the block that contains $i+1$, so $i+1$ appears to the right of $i$. Thus $i$ \emph{is not} an inverse descent of $\pi$.
\item If $w_i w_{i+1} \in \{ba, ca, cb\}$, then the same reasoning as the previous case shows that $i$ \emph{is} an inverse descent of $\pi$.
\end{itemize}
Therefore, $i$ is an inverse descent of $\pi$ if and only if $w_i w_{i+1} \in \{ba, bb, ca, cb\}$.

Finally, $i$ is the starting position of an occurrence of $321$ in $\pi^{-1}$ if and only if $i$ and $i+1$ are inverse descents of $\pi$, if and only if $w_i w_{i+1}$ and $w_{i+1} w_{i+2}$ are both in $\{ba, bb, ca, cb\}$, if and only if $w_i w_{i+1} w_{i+2} \in \{bba, bbb, cba, cbb\}$.
\end{proof}

Recall that $W_n$ was defined to be the set of words in $A^n$ that satisfy the conditions from both Lemma \ref{lem:canonicalword} and Lemma \ref{lem:avoidsubword}. Hence, Lemmas \ref{lem:canonicalword} and \ref{lem:avoidsubword} imply that $\phi$ restricted to $N'_n$ is a bijection onto $W_n$. This completes the proof of Proposition \ref{lem:NW}.

\subsection{Proof of Proposition \ref{lem:regexmain}}

Our proof of Proposition \ref{lem:regexmain} relies on the following lemma.

\begin{lem} \label{lem:cva}
Let $v$ be a word that does not end with $a$. If $cva$ avoids the subwords $bba$, $bbb$, $cba$, and $cbb$, then $v$ matches the regular expression $(c \cup bc \cup a^+ b \cup a^+ c)^*$.
\end{lem}

\begin{proof}
We induct on the length of $v$. The empty word matches the regular expression, so let us assume that $v$ is not empty.

First, suppose $v$ starts with $c$. Let us write $v = cv'$ for a word $v'$. Since $cva = ccv'a$ avoids the given subwords, so does $cv'a$, and since $v$ does not end with $a$, neither does $v'$. Thus, by the induction hypothesis $v'$ matches $(c \cup bc \cup a^+ b \cup a^+ c)^*$, and so $v = cv'$ also does.

Next, suppose $v$ starts with $b$. If $v = b$, then $cva = cba$ would be a forbidden subword, so $v$ has length at least $2$. Since $cv$ starts with $cb$, it must start with $cbc$---otherwise $cv$ would start with a forbidden subword $cbb$ or $cba$. Let us write $v = bcv'$ for a word $v'$. Since $cva = cbcv'a$ avoids the given subwords, so does $cv'a$, and since $v$ does not end with $a$, neither does $v'$. Thus, by the induction hypothesis $v'$ matches $(c \cup bc \cup a^+ b \cup a^+ c)^*$, and so $v = bcv'$ also does.

Finally, suppose $v$ starts with $a$. Let $i$ be the largest number such that $v$ starts with $a^i$. We know that $i$ is not the length of $v$, because $v$ does not end with $a$. So the $(i+1)$th letter of $v$ is either $b$ or $c$; without loss of generality, suppose it is $b$. Write $v = a^i bv'$ for a word $v'$. Since $cva = ca^i bv'a$ avoids the given subwords, so does $bv'a$, and since $v$ does not end with $a$, neither does $v'$. Thus, by the induction hypothesis $v'$ matches $(c \cup bc \cup a^+ b \cup a^+ c)^*$, and so $v = a^ib v'$ also does because $i\ge1$.
\end{proof}

We are now ready to prove Proposition \ref{lem:regexmain}.

\begin{proof}[Proof of Proposition \ref{lem:regexmain}]
Suppose $w$ matches the regular expression $a^* c \,(c \cup bc \cup a^+ b \cup a^+ c)^*\, a^+ c^*$. Clearly $w = a^i c u ac^j$ for some $i,j\geq 0$ and some word $u$, and a routine case-checking argument verifies that $w$ avoids $bba$, $bbb$, $cba$, and $cbb$. Therefore, $w \in W_n$ (where $n$ is the length of $w$).

Now, suppose $w \in W_n$, and write $w = a^i c u ac^j$. Since the prefix $a^i c$ matches $a^* c$ and the suffix $ac^j$ matches $ac^*$, it remains to prove that $u$ matches $(c \cup bc \cup a^+ b \cup a^+ c)^*\, a^*$. Let $l$ be the largest number such that $u$ ends with $a^l$ (so $l=0$ if $u$ does not end with $a$). We can write $u = va^l$, where $v$ is a word that does not end with $a$. Since $a^l$ matches $a^*$, and because Lemma \ref{lem:cva} implies $v$ matches $(c \cup bc \cup a^+ b \cup a^+ c)^*$, we conclude that $u = va^l$ matches $(c \cup bc \cup a^+ b \cup a^+ c)^*\, a^*$ as desired.
\end{proof}

\subsection{Proof of Proposition \ref{lem:Zfib}} \label{sec:ProofProp14}

A classical combinatorial model for Fibonacci numbers is given by monomino-domino tilings of rectangles, that is, $f_k$ is the number of tilings of a $1 \times k$ rectangle with $1 \times 1$ blocks (monominoes) and $1 \times 2$ blocks (dominoes). Furthermore, $f_{k-1}$ is the number of such tilings in which the leftmost block is a monomino. Thus, $f_{k-1} f_{k}$ is the number of ordered pairs of these tilings in which the first tiling begins with a monomino (and with no restrictions on the second tiling). By stacking the first tiling on top of the second, we get a monomino-domino tiling of a $2 \times k$ rectangle---where all dominoes are placed horizontally---with a monomino in the top-left corner. Let $T_k$ be the set of these tilings; then $|T_k| = f_{k-1} f_k$.

To complete the proof of Proposition \ref{lem:Zfib}, we will construct a bijection between $Z_k$ and $T_k$. Recall that $Z_k$ is the set of words in $A^k$ that match the regular expression $a^* c \,(c \cup bc \cup a^+ b \cup a^+ c)^*$, and observe that this regular expression is unambiguous: every word in $Z_k$ has a unique decomposition into a subword matching $a^* c$ followed by a sequence of subwords matching $c$, $bc$, $a^+b$, or $a^+c$. Equivalently, this is a decomposition into a sequence of subwords matching $c$, $bc$, $a^+b$, or $a^+c$ such that the first subword matches $c$ or $a^+c$. We now map a word $z$ in $Z_k$ to a tiling in $T_k$ by mapping the subwords of $z$ in its decomposition to indecomposable horizontal segments of the tiling as follows:
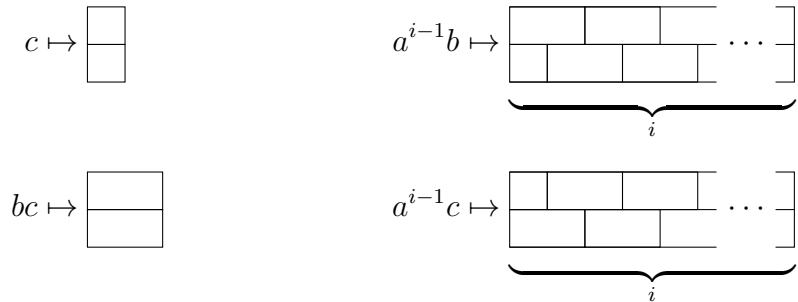
\begin{align*}
c &\mapsto \begin{tikzpicture}[scale=.5,baseline=1pc]
\draw (0,0) rectangle (1,1);
\draw (0,1) rectangle (1,2);
\end{tikzpicture} 
&
a^{i-1} b &\mapsto \underbrace{\begin{tikzpicture}[scale=.5,baseline=1pc]
\draw (0,0) rectangle (1,1);
\draw (1,0) rectangle (3,1);
\draw (3,0) rectangle (5,1);
\draw (0,1) rectangle (2,2);
\draw (2,1) rectangle (4,2);
\draw (0,0) -- (5.5,0);
\draw (0,1) -- (5.5,1);
\draw (0,2) -- (5.5,2);
\node at (1,-0.1) {};
\end{tikzpicture}
\,\cdots\,
\begin{tikzpicture}[scale=.5,baseline=1pc]
\draw (0,0) -- (0.5,0) -- (0.5,2) -- (0,2);
\draw (0,1) -- (0.5,1);
\end{tikzpicture}}_i
\\[.7pc]
bc &\mapsto \begin{tikzpicture}[scale=.5,baseline=1pc]
\draw (0,0) rectangle (2,1);
\draw (0,1) rectangle (2,2);
\end{tikzpicture}
&
a^{i-1} c &\mapsto \underbrace{\begin{tikzpicture}[scale=.5,baseline=1pc]
\draw (0,0) rectangle (2,1);
\draw (2,0) rectangle (4,1);
\draw (0,1) rectangle (1,2);
\draw (1,1) rectangle (3,2);
\draw (3,1) rectangle (5,2);
\draw (0,0) -- (5.5,0);
\draw (0,1) -- (5.5,1);
\draw (0,2) -- (5.5,2);
\node at (1,-0.1) {};
\end{tikzpicture}
\,\cdots\,
\begin{tikzpicture}[scale=.5,baseline=1pc]
\draw (0,0) -- (0.5,0) -- (0.5,2) -- (0,2);
\draw (0,1) -- (0.5,1);
\end{tikzpicture}}_i
\end{align*}
(Here, $i \ge 2$. Note that for the tilings corresponding to $a^{i-1}b$ and $a^{i-1}c$, the right end of the tiling depends on whether $i$ is even or odd.) The condition that each word in $Z_k$ starts with $c$ or $a^+c$ is equivalent to the condition that each tiling in $T_k$ has a monomino in the top-left corner. It is clear that this mapping is bijective, and we have established Proposition \ref{lem:Zfib}.

For example, if $z = aacbcccaaabbcac \in Z_{15}$, then $z$ decomposes as $aac|bc|c|c|aaab|bc|ac$, so $z$ maps to
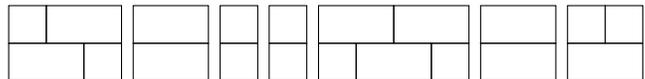
\[ \begin{tikzpicture}[scale=.5,baseline=1pc]
\draw (0,0) rectangle (2,1);
\draw (2,0) rectangle (3,1);
\draw (0,1) rectangle (1,2);
\draw (1,1) rectangle (3,2);
\end{tikzpicture}\,\,
\begin{tikzpicture}[scale=.5,baseline=1pc]
\draw (0,0) rectangle (2,1);
\draw (0,1) rectangle (2,2);
\end{tikzpicture}\,\,
\begin{tikzpicture}[scale=.5,baseline=1pc]
\draw (0,0) rectangle (1,1);
\draw (0,1) rectangle (1,2);
\end{tikzpicture}\,\,
\begin{tikzpicture}[scale=.5,baseline=1pc]
\draw (0,0) rectangle (1,1);
\draw (0,1) rectangle (1,2);
\end{tikzpicture}\,\,
\begin{tikzpicture}[scale=.5,baseline=1pc]
\draw (0,0) rectangle (1,1);
\draw (1,0) rectangle (3,1);
\draw (3,0) rectangle (4,1);
\draw (0,1) rectangle (2,2);
\draw (2,1) rectangle (4,2);
\end{tikzpicture}\,\,
\begin{tikzpicture}[scale=.5,baseline=1pc]
\draw (0,0) rectangle (2,1);
\draw (0,1) rectangle (2,2);
\end{tikzpicture}\,\,
\begin{tikzpicture}[scale=.5,baseline=1pc]
\draw (0,0) rectangle (2,1);
\draw (0,1) rectangle (1,2);
\draw (1,1) rectangle (2,2);
\end{tikzpicture}.
\]

Our bijective proof of Theorem \ref{t-ilpk} involved several different bijections; let us now describe the composite bijection. Since we had to partition $W_n$ into subsets $W_n^{j,k}$, our bijection must keep track of $j$ and $k$ as well as the resulting tiling. Thus, we are really mapping $N'_n$ onto the set of ordered triples $(j,k,\tau)$ such that $1 \le k \le n-1$, $0 \le j \le n-k-1$, and $\tau \in T_k$. The number of such ordered triples is the right side of Equation \eqref{eqn:fibonaccisum}. We illustrate the bijection using the example $\pi = 1\,2\,8\,9 \,10 \,14\,16\,17\,|\,12\, 11 \,4\,|\,3\,5\,6\,7 \,13 \,15\,18\,19\,20 \in N'_{20}$. We have $\phi(\pi) = aacbcccaaabbcacaaccc \in W_{20}$, and in fact $aacbcccaaabbcac|aa|ccc \in W^{3,15}_{20}$ due to the prefix $aacbcccaaabbcac \in Z_{15}$ and the suffix $a^2 c^3$. Finally, $aacbcccaaabbcac$ maps to the tiling $\tau$ that we saw in the previous example, so $\pi$ maps to $(3,15,\tau)$.

\subsection{Remarks on avoiding \texorpdfstring{$m \cdots 21$}{m...21}}

Parts of the construction we have described can be generalized to permutations avoiding the consecutive pattern $m \cdots 21$, for any $m \ge 3$. We outline the more general construction here but omit the details.

Let $W_n^{(m)}$ denote the set of words $w$ in $A^n$ of the form $w = a^i cuac^j$, where $i,j \ge 0$ and $u \in A^{n-i-j-2}$, such that $w$ avoids the subwords $b^{m-1}a$, $b^m$, $cb^{m-2}a$, and $cb^{m-1}$. Then the map $\phi \colon N_n \to A^n$ restricts to a bijection between 
$\{\, \pi \in N_n : \pi^{-1} \in \mathfrak{S}_{n}(m \cdots 21) \,\}$ and $W_n^{(m)}$; this is a generalization of Proposition \ref{lem:NW}.

Next, we can generalize Proposition \ref{lem:regexmain} to the fact that, for $w \in A^n$, we have $w \in W_n^{(m)}$  if and only if $w$ matches the regular expression
\begin{equation} a^* c\,{\left[ b^{\le m-3 } (c \cup bc \cup a^+ b \cup a^+ c )\right]}^*\,b^{\le m-3} a^+ c^*, \label{eqn:regexgeneral} \end{equation}
where $b^{\le t} \coloneqq (\varepsilon \cup b \cup b^2 \cup \cdots \cup b^t)$. Setting $m=3$ recovers the regular expression from Proposition \ref{lem:regexmain}. 

It can be shown that $x^{2}(x^{m-2}-1)/((1-x)^{2}(x^{m+1}-3x^{m}+3x-1))$ is the generating function for words matching the regular expression in \eqref{eqn:regexgeneral}; this generating function appeared in our first proof of Theorem \ref{t-ilpk}.

Unfortunately, the words matching the regular expression in \eqref{eqn:regexgeneral} do not have a nice interpretation in terms of tilings that would generalize the mapping from Section \ref{sec:ProofProp14}. We have found a few ways of turning these words into tilings of $2 \times k$ rectangles, but none are as natural-seeming as the one for $m=3$; more importantly, none can be broken down as a one-row tiling stacked on top of another one-row tiling, so these would not yield a simple product formula like $f_{k-1} f_k$ in the case of $m=3$. The best we can do is to use \eqref{eqn:regexgeneral} to derive a rational generating function, but there does not seem to be anything of great significance about this generating function or the associated sequence of numbers.
\\

\noindent \textbf{Acknowledgements.} The authors thank an anonymous referee for their suggestions on improving the presentation of this paper. The second author was partially supported by an AMS-Simons Travel Grant.

\bibliographystyle{plain}
\bibliography{bibliography}

\providecommand\noopsort[1]{}
\begin{thebibliography}{10}

\bibitem{Benjamin2003}
Arthur~T. Benjamin and Jennifer~J. Quinn.
\newblock {\em Proofs that {R}eally {C}ount: {T}he {A}rt of {C}ombinatorial
  {P}roof}, volume~27 of {\em The Dolciani Mathematical Expositions}.
\newblock Mathematical Association of America, Washington, DC, 2003.

\bibitem{Bevan}
David Bevan.
\newblock Growth rates of permutation grid classes, tours on graphs, and the
  spectral radius.
\newblock {\em Trans. Amer. Math. Soc.}, 367:5863--5889, 2015.

\bibitem{Flajolet2009}
Philippe Flajolet and Robert Sedgewick.
\newblock {\em Analytic {C}ombinatorics}.
\newblock Cambridge University Press, 2009.

\bibitem{Foulkes1976}
H.~O. Foulkes.
\newblock Enumeration of permutations with prescribed up-down and inversion
  sequences.
\newblock {\em Discrete Math.}, 15(3):235--252, 1976.

\bibitem{Gessel1984}
Ira~M. Gessel.
\newblock Multipartite {$P$}-partitions and inner products of skew {S}chur
  functions.
\newblock {\em Contemp. Math.}, 34:289--317, 1984.

\bibitem{Gessel2018}
Ira~M. Gessel and Yan Zhuang.
\newblock Shuffle-compatible permutation statistics.
\newblock {\em Adv. Math.}, 332:85--141, 2018.

\bibitem{HuczynskaVatter}
Sophie Huczynska and Vincent Vatter.
\newblock Grid classes and the {Fibonacci} dichotomy for restricted
  permutations.
\newblock {\em Electron. J. Combin.}, 13:R54, 14 pp., 2006.

\bibitem{MurphyVatter}
Maximillian~M. Murphy and Vincent Vatter.
\newblock Profile classes and partial well-order for permutations.
\newblock {\em Electron. J. Combin.}, 9(2):R17, 30 pp., 2003.

\bibitem{Sipser}
Michael Sipser.
\newblock {\em Introduction to the {T}heory of {C}omputation}.
\newblock Cengage Learning, 3rd edition, 2012.

\bibitem{oeis}
N.~J.~A. Sloane.
\newblock The {O}n-{L}ine {E}ncyclopedia of {I}nteger {S}equences.
\newblock Published electronically at \url{http://oeis.org}.

\bibitem{Stanley2001}
Richard~P. Stanley.
\newblock {\em Enumerative {C}ombinatorics, Vol. 2}.
\newblock Cambridge University Press, {\noopsort{b}}2001.

\bibitem{Zhuang2021}
Yan Zhuang.
\newblock A lifting of the {G}oulden--{J}ackson cluster method to the
  {M}alvenuto--{R}eutenauer algebra.
\newblock {\em Algebr. Comb.}, 5(6):1391--1425, 2022.

\end{thebibliography}

\end{document}